\documentclass[12pt]{article}

\usepackage{amsmath,amsthm,amssymb,graphicx}

\def\N{\mathbb N}
\def\N0{\mathbb N_0}

\def\real{\mathbb R}
\def\complex{\mathbb C}

\def\H2{\mathbb H^2}

\def\L2per{\mathbb L^2_{per}}
\def\BS{\partial \strip} 

\def\dom{\mathcal D}

\def\strip{\Omega}

\def\proj{\mathcal P}

\def\dop{{\mathcal L}}

\title{\bf Eigenfunctions \\ 
of Periodic Differential Operators \\
Analytic in a Strip}
\author{Robert Carlson \\
Department of Mathematics \\ 
University of Colorado at Colorado Springs \\
carlson@math.uccs.edu}

\newtheorem{thm}{Theorem}[section]

\newtheorem{lem}[thm]{Lemma}
\newtheorem{prop}[thm]{Proposition}

\theoremstyle{definition}

\theoremstyle{remark}

\newcommand{\thmref}[1]{Theorem~\ref{#1}}

\newcommand{\lemref}[1]{Lemma~\ref{#1}}
\newcommand{\propref}[1]{Proposition~\ref{#1}}

 \numberwithin{equation}{section}

\begin{document}

\maketitle

\begin{abstract}
Ordinary differential operators with periodic coefficients analytic in a strip act on a Hardy-Hilbert space
of analytic functions with inner product defined by integration over a period on the boundary of the strip.
Simple examples show that eigenfunctions may form a complete set for a narrow strip, but completeness
may be lost for a wide strip.  Completeness of the eigenfunctions in the Hardy-Hilbert space is established for regular second order operators
with matrix-valued coefficients when the leading coefficient satisfies a positive real part condition throughout the strip.

\end{abstract}


\vskip 10pt


\newpage

\section{Introduction}

Suppose $\dop = P_2(z)D^2 + P_1(z)D + P_0(z) $ is an ordinary differential operator
whose coefficients $P_j(z)$ are $K \times K$ matrix valued, $2\pi $ - periodic, and complex analytic in a strip $\strip $
of height $T$,
\[\strip = \{  z = x + i\tau \in \complex,  -T < \tau < T \}.\]
In many important examples $\dop $ induces a self adjoint operator on the Hilbert space $\oplus _K \L2per$, the $2\pi $ periodic
vector-valued functions which are square integrable on $[0,2\pi ]$.
The associated spectral theory has been polished and elaborated since the work of Sturm and Liouville in the early nineteenth century.
However the assumption of analyticity of the coefficients suggests viewing the spectral theory of $\dop $ in a different light, with fresh opportunities and challenges.

Considering the opportunities first, functions analytic in a strip have Fourier series whose coefficients decay at an exponential rate.
This property can have beneficial consequences for numerical calculations \cite{Tadmor}.
There is also a convenient Hilbert space of $2\pi $ periodic analytic functions, the periodic Hardy space $\H2 $, with strong links to Fourier series and exponential decay rates, 
which provides an operator theoretic context in which to study $\dop $.  Operators $\dop $ which are selfadjoint on $\oplus _K \L2per$ will typically be nonnormal
as operators on $\oplus _K \H2$.  The analysis of nonnormal operators, long a challenge in differential equations, is once again drawing attention \cite{Davies07, Tref}.  
 
Of course the loss of selfadjointness in the Hardy space setting is also a challenge.
Despite long running efforts \cite{CL, DS3, Davies07,Tref} dating back at least to G.D. Birkhoff \cite{Birkhoff}, 
there is no comprehensive spectral theory for nonnormal differential operators.
Problems with familiar features are often closely linked to selfadjoint or normal operators \cite[p. 298-313]{CL}, \cite[p. 2290-2374]{DS3}. 
When the link is too weak, unfamiliar behavior is possible.  There 
are simple examples \cite{Carlson80, Seeley} whose the spectrum comprises the entire complex plane.  
In other cases, when there is a discrete spectrum, the eigenfunctions may not be complete, and the operators may behave badly 
with respect to perturbations \cite{Davies01}.

This work focuses on eigenfunction completeness in $\oplus _K \H2 $.
Scalar operator ($K=1$) examples with entire coefficients demonstrate that eigenfunction completeness in $\H2 $ can hold on thin strips, but fail as the strip $\strip $ gets too wide.
These examples, which take advantage of a positive real part condition to extend a classical change of variables to a conformal map,
suggest looking for positive results when  $\dop $ generates a semigroup on $\oplus _K \H2 $. 
The main positive result \thmref{mainthm} establishes $\oplus _K \H2 $ completeness of the eigenfunctions for operators $\dop $ which are self adjoint and positive on $\oplus _K \L2per$
with a sectorial continuation to $\strip $.  The proof combines semigroup ideas with a variety of eigenfunction estimates. 

The subsequent sections begin with a brief review of relevant aspects of the Hardy space $\H2 $.
There is a vast literature on this space and its Banach space relatives, although the spatial domain is usually the unit disc or
a half-space \cite{Duren, Garnett, Hoffman}.  A change of variables converts our periodic problems on a strip to equivalent ones
on an annulus. Hardy spaces on the annulus were studied in \cite{Sarason}.  A recent reference on the infinite strip is \cite{BK}.

After treating some general operator theoretic issues, scalar examples are considered.
An analytic extension of a familiar change of independent variables plays an important role.
Some first order examples with entire coefficients and a discrete spectrum exhibit a threshold phenomenon; there is  
a complete set of eigenfunctions for $\H2 $ if the strip height $T$ is below the threshold, but completeness fails for greater heights.

In the scalar setting the conformal change of variables is also effective for second order operators;
since $D^2$ is selfadjoint on $\H2$, previously established perturbation techniques 
can be used for operators in the Liouville normal form $D^2 + q(z)$.
Other techniques are required for operators with matrix coefficients since the change of variables is less effective.
This is where the semigroup ideas and a variety of eigenfunction estimates come into play.

Although there is a recent book \cite{G3} treating evolution equations with a complex spatial variable,
this work actually started as a reaction to an old paper \cite{Villone} by A. Villone based on his thesis, written
under the direction of E.A. Coddington.  This work and subsequent extensions 
provide a thorough and somewhat surprising analysis of selfadjoint differential operators (and so, implicitly,
certain evolution equations) acting on a Hilbert space of analytic functions on the unit disc, 
with inner product given by integration with respect to area measure.
      
It is a pleasure to acknowledge an early assist from Don Marshall, who pointed out a simple proof of \propref{confmap}.

\section{Preliminary material}

\subsection{The periodic Hardy space $\H2 $ }

The Hardy spaces provide settings where complex analysis, Fourier analysis, and functional analysis have a productive interaction.
The textbooks \cite{Duren, Garnett, Hoffman} focus on the unit disc as the main example, but
the annulus and infinite strip have also received attention \cite{BK,Sarason}.  
Since the main interest here is the study of analytic extensions of $2\pi $ periodic differential operators to a complex domain 
$\strip = \{  z = x + i\tau ,  -T < \tau < T \}$,
the Hardy-Hilbert space $\H2 $ seems like a natural environment.      
 
As an introduction to the Hardy space $\H2 $, consider a Fourier series
\[f(z) = \sum_{n=-\infty}^{\infty} c_n \exp( i nz)
= A_0 + \sum_{n=1}^{\infty} [A_n\cos (nz) + B_n\sin(nz)]\]
Writing $z= x +  i  \tau $,
\[ f(x + i \tau ) = \sum_{n=-\infty}^{\infty} c_n \exp(-  n\tau ) \exp( i nx), \]
so the series will converge to a function analytic in a strip if the coefficients $c_n$ have sufficiently rapid decay.  
$\H2 $ can be defined as the set of $2\pi $ periodic  functions analytic in $\strip $ satisfying the weighted summability condition
\begin{equation} \label{wtcond1}
\| f \| _H^2 = \sum_{n=-\infty}^{\infty} |c_n|^2 \cosh (2 n T) < \infty ,
\end{equation}
for the Fourier coefficients $c_n$.  This condition 
ensures uniform convergence of the series on any strip with height $ T_1 < T $.

$\H2 $ has an inner product 
\begin{equation} \label{H2ip}
\langle f,g \rangle _H 
= \frac{1}{4\pi }\int_0^{2\pi }[f(x + i T ) \overline{g(x+ iT )}  + f(x - i T ) \overline{g(x - iT )}] \ d x ,
\end{equation}
which will also be denoted as
\[\langle f,g \rangle _H = \int_{\partial \strip} f(z)\overline{g(z)} .\] 

The exponentials $e^{inz}$ are orthogonal in $\H2 $, with $\| e^{ i nz} \| ^2 = \cosh (2 n T)$.
Under the change of variables 
\begin{equation}  \label{annsect}
w = \exp (iz), \quad  F(w) = f(-i \log(w)),
\end{equation} 
$\H2 $ becomes a Hilbert space of functions analytic on an annulus $A = \{ e^{-T} < |w| < e^T \}$.

The following observations (with abbreviated proofs) are standard. 

\begin{prop} \label{Hbasics}
If $f \in \H2 $, then the boundary values $f(x \pm iT)$ are well-defined elements of $\L2per$,
and the set $\{ (f(x + iT),f(x - iT)), f \in H^2 \} $ is a closed subspace of $\L2per \oplus \L2per$.
$H^2$ is the set of all functions analytic in the strip with period $2\pi $ and 
\begin{equation} \label{H2bnd2}
\sup_{-T < \tau < T }\int_0^{2\pi} |f(x + i \tau  )|^2 \ dx  < \infty .
\end{equation}
\end{prop}

\begin{proof}

The condition \eqref{wtcond1} implies that 
\[\int_0^{2\pi} |f(x \pm iT)|^2 \ dx = \sum_{n = -\infty}^{\infty} |c_n|^2 e^{\mp 2 nT} < \infty ,\]
so the boundary values $f(x \pm iT)$ are well-defined elements of $\L2per$.

Suppose $\{ g_k \}$ is a Cauchy sequence in $\H2 $ with Fourier coefficients $g_k(n)$.
The functions $g_k(n)$ converge pointwise to some $g(n)$ which satisfies \eqref{wtcond1},
so the space of sequences satisfying \eqref{wtcond1} is complete.
By \eqref{H2ip} the map $\{c_n \} \to  (f(x + iT),f(x - iT))$ is an isometry, so the image is closed in
$\L2per \oplus \L2per$.

Suppose $g(z)$, like every $f \in \H2 $,  is analytic in the strip, has period $2\pi $ and 
\[ \sup_{-T < \tau < T }\int_0^{2\pi} |g(x+ i \tau  )|^2 \ dx < \infty .\]
Taking advantage of the change of variables in \eqref{annsect},
the Fourier series for $g(z)$ and the Laurent expansion for $G(w) = g(-i \log(w))$
have the same coefficients.  Both converge uniformly on compact subsets of their respective domains,
the open strip $\strip $ and annulus $A$.
The bound \eqref{H2bnd2} gives
\[\sup_{-T < \tau < T } \sum_{n=-\infty}^{\infty} |c_n|^2e^{2 n \tau }  < \infty ,\]
implying \eqref{wtcond1}.

\end{proof}

It is also easy to verify that evaluation at a point $z \in \strip $ is a continuous linear functional on $\H2 $. 

 \begin{prop} \label{Uest}
A function $f(z) \in \H2 $ satisfies 
\[|f(x +i \tau )| \le 2\| f \| _H  (\frac{1}{1 -  \exp(2( |\tau | - T))} )^{1/2}  .\]
Consequently, there is a constant $C$ such that uniformly in $x$,
\[ \int_{-T}^T |f(x+ i\tau )| \ d\tau \le C \| f \| _2 .\]

\end{prop}

\begin{proof}
 
The Cauchy-Schwarz inequality gives the estimate
\[ |f(x + i \tau )| \le \sum_n |c_n| \exp( | n \tau |) 
 = \sum_n |c_n| \exp( | n| T ) \exp( | n| ( |\tau | - T) )  \]
 \[ \le  (\sum_n |c_n|^2 \exp(2 | n| T ) )^{1/2} (\sum_n \exp(2 | n| ( |\tau | - T) )^{1/2}  \]
\[\le 2\| f \| _H  (\sum_{n=0}^{\infty}  \exp(2 n ( |\tau | - T) )^{1/2}  .\]
Summing the geometric series, 
\[|f(\sigma +i \tau )| \le 2\| f \| _H  (\frac{1}{1 -  \exp(2 ( |\tau | - T))} )^{1/2}  \]

For $T - | \tau | $ close to zero,
\[|1 -  \exp(2 ( |\tau | - T))| \simeq 2 (T - |\tau | ),\]
allowing integration with respect to $\tau $.

\end{proof}

\begin{lem}
For $f \in \H2 $, evaluation at $w \in \strip $ is given by inner product with
\begin{equation} \label{pteval1}
g_w(z) = \sum_n \frac{\overline{e^{inw}}}{\sqrt{\cosh (2nT)} } \frac{e^{inz}}{\sqrt{\cosh (2nT)}}.
\end{equation}
The map $w \to g_w$ is analytic from $\strip $ to $\H2$, and for $v,w \in \strip $
\begin{equation} \label{pteval2}
 \| g_v - g_w \| ^2_H \le 4 |v-w|^2 \sum_{n=1}^{\infty} ne^{n[\tau - T]} = 4|v-w|^2 \frac{e^{\tau - T}}{(1 - e^{\tau - T})^{2}}  .
\end{equation} 
\end{lem}

\begin{proof}
Note that for $w \in \strip $ the coefficients of $g_w(z)$ are square summable.
Expanding $f(z) \in \H2$ in the orthonormal basis of exponentials
\[f(z) = \sum_n c_n \frac{e^{inz}}{\sqrt{\cosh (2nT)}},\]
leads to 
\[\langle f,g_w \rangle = \sum_n c_n\frac{e^{inw}}{\sqrt{\cosh (2nT)}} = f(w).\]

For $v,w \in \strip $,
\[ \| g_v - g_w \| ^2_H = \sum_n \frac{|e^{inv} - e^{inw}|^2}{\cosh (2nT) } . \]
Suppose $\tau = \max (|\Im (v)|, |\Im (w)|)$.  Then
\[ |e^{inv} - e^{inw}| = | \int_w^v in e^{ins} \ ds | \le |v-w| |n| e^{|n|\tau },\] 
and estimating with the derivative of the geometric series gives the continuity estimate
\[ \| g_v - g_w \| ^2_H \le 4 |v-w|^2 \sum_{n=1}^{\infty} ne^{n[\tau - T]} = 4|v-w|^2 \frac{e^{\tau - T}}{(1 - e^{\tau - T})^{2}}  .\]   
\end{proof}

\subsection{Differential operators}

For $j = 0,\dots ,N$, suppose $P_j(z)$ is a $K \times K$ matrix valued function with entries in $\H2 $.
The differential operator
\begin{equation} \label{opdef}
\dop = P_N(z) D^N + \dots + P_1(z)D + P_0(z), \quad D = \frac{d}{dz}
\end{equation}
may be considered as an operator on the column vector valued Hilbert spaces
$\oplus _K \L2per$ or $\oplus _K \H2 $, with respective inner products
\[ \langle f,g \rangle _L = \frac{1}{2\pi }\int_0^{2\pi} g^*(x)f(x) \ dx , \quad \langle f,g \rangle _H = \int_{\BS}  g^*(z)f(z) \ dx.\]
The corresponding Hilbert space norms will be denoted $\| f \| _L$ and $\| g \|_H$.

For $Z  = (z_1,\dots ,z_K)^T \in \complex ^K$, let $Z^* = (\overline{z_1}, \dots ,\overline{z_K})$
and $\| Z \| _E = (Z^*Z)^{1/2}$.  The conjugate transpose on a matrix $P_j$ is $P_j^*$.   
$E_1, \dots , E_K$ will denote the standard basis for $\complex ^K$,
and $I_K$ will be the $K \times K$ identity matrix.

$\dop $ is densely defined in $ \oplus _K \H2 $ 
on the domain $\dom _0$ which is the linear span of $\{ e^{2\pi i nz}E_k \} $.
A larger domain for an operator $\dop $ is 
\[\dom = \{ f \in \oplus _K \H2, \ \dop f \in \H2\} .\]

\begin{prop}
The operator $\dop $ with domain $\dom $ is closed in $H^2$.
\end{prop}

\begin{proof}
The argument follows \cite{Villone}. 
Suppose for  $j = 1,2,3,\dots $ that $f_j \in \dom $, and that $\{ f_j \}$ and $\{ \dop f_j \}$ are Cauchy sequences in $\H2 $,
converging respectively to $f$ and $g$.

By \propref{Uest}, on any compact subset $K$ of $\strip $ the sequence $\{ f_j \}$ converges uniformly to $f$.
By the Cauchy integral formula the same convergence applies to the derivatives of $f_j$, so 
$\{ \dop f_j \} $ converges uniformly on $K$, with $\dop f = g$.  
 
\end{proof}

Say that $\dop $ is regular if, for $j = 0,\dots ,N$, the coefficients $P_j(z)$ have $j$ continuous derivatives, 
with $\det P_N(z) \not= 0$, for all $z$ in the closed strip $ \overline{\strip }$.
Regular operators $\dop $ may be defined on the same domain in $\oplus_K \H2 (\strip)$ as the diagonal operator 
\[I_K D^N = \begin{pmatrix} D^N & 0 & \dots & 0 \cr
0 & D^N & \dots & 0 \cr
\vdots & \vdots & \dots & \vdots \cr 
0 & 0 & \dots & D^N 
\end{pmatrix}. \]
The scalar operator $D^N$ is a normal operator with compact resolvent on the domain consisting those 
functions  whose Fourier coefficients satisfy
\begin{equation} \label{domdef}
f \in \H2, \quad  \| f^{(N)} \|_H^2 =  \sum_{n=-\infty}^{\infty} |c_n|^2 n^{2N} \cosh (2 n T) < \infty .
\end{equation}
These comments extend easily to the case of vector valued functions and matrix operator coefficients. 

$\dop$ may also be interpreted as an operator on $\oplus _K \L2per$, the usual Hilbert space of $2\pi $ periodic square integrable vector valued functions
on the real line, with inner product
\[\langle f,g \rangle _L = \frac{1}{2\pi } \int_0^{2\pi } g^*(x) f(x)\ dx .\]

\begin{lem} \label{pereig}
Suppose $\dop $ is regular on $\strip $.
The periodic eigenvalues $\lambda _m$ of $\dop $ on $\H2 $ are the same as those for $\dop $ on 
$\L2per$.
\end{lem}

\begin{proof}
The solutions $y(z,\lambda )$ of the differential equation $\dop y = \lambda  y$ are \cite[p. 90-91]{CL} well defined analytic functions in $\strip $.
If $y(z+2\pi ,\lambda ) = y(z,\lambda )$ for $z \in \strip _T$, then the restriction of $y$ to $z$ real is a $2\pi $-periodic eigenfunction.
If $y(x,\lambda )$ is an eigenvalue in the context of $\L2per$, then 
$y(z+2\pi ,\lambda ) - y(z,\lambda ) = 0$ for $z \in \real $, and this identity extends by analyticity to $z \in \strip $.  
\end{proof}

Let $R_1 = \oplus _K \L2per$.  As noted in \propref{Hbasics}, restriction of $f \in \oplus _K \H2 $
to its boundary values lets us interpret $\oplus _K \H2 $ as a closed subspace of  $R_2 = R_1 \oplus R_1$.
Regular operators $\dop $, with their extensive classical theory \cite{CL}, \cite[pp. 1278-1333]{DS2}, act on this space by
\[\dop (f(x + iT),f(x-it)) = (\dop (x + iT)f(x+iT), \dop (x-iT) f(x-iT)) .\] 
In this setting, $\dop $ acting on $\H2 $ is the restriction of the differential operator to an invariant subspace of infinite codimension.
When the resolvent set is not empty, regular operators have compact resolvents on $R_2$; this property is inherited
by $\dop $ acting on $\oplus _k \H2$.

Viewed as an operator on $R_2$, $\dop $ has a classical adjoint which acts by 
\[\dop ^+ = \sum_{j=0}^N (-1)^jD^jP_j^*.\]
Viewed as an operator on $\oplus _K \H2 $, the (graph of the) adjoint operator $L^*$ is the
set of pairs $(g,h) \in [\oplus _K \H2 ] \oplus [\oplus _K \H2 ]$ such that
$ \langle \dop f,g \rangle _H = \langle f,h \rangle _H $
for all $f$ in the domain of $\dop $ in $\oplus _K \H2 $.
Let $\proj $ denote the orthogonal projection of $R_2$ onto $\H2$. 
Since the inner product formulas for $\oplus \H2 $ and $R_2$ agree, 
the adjoint acts by $L^*g = \proj L^+g$.

\section{Basic examples}

Basic examples of first and second order regular operators $\dop $ acting on $\H2 $
exhibit two rather different behaviors.  In some cases a standard change of variables extends to
a conformal map which transforms the highest order term of $\dop $ to the skew adjoint operator $D$ or the 
self adjoint operator $D^2$.  In these cases the eigenfunctions have dense span in 
$\H2 $ for some strip contained in the image of the conformal map.  When the 
strip is sufficiently wide, the extended change of variables may fail to be one-to-one.
In such cases the eigenfunctions, which have dense span in $\L2per$, 
may have a closed span of infinite codimension in $\H2 $.

Perhaps the simplest examples are the operators
\begin{equation} \label{firstex}
\dop _1 = e^{i \phi(z)}D e^{-i \phi (z)} = \frac{d}{dz} - i \phi ' (z).
\end{equation}
Suppose $\phi (z)$ is entire, $2\pi $ periodic, and real-valued for $z \in \real $.
The operator $\dop _1$ is skew adjoint and unitarily equivalent to $D$ on $\L2per$.  
On $\H2 $ the operator $D$ is still skew adjoint, but now $\dop _1$ is similar to $D$
with eigenfunctions $\psi _n(z) = e^{i \phi (z)}e^{i nz}$.
Since $\phi (z)$ is entire, the operator similarity holds for any strip height $T$.

For more general first order operators
\[\dop _1 = p_1(z) D + p_0(z),\]
a common real variable technique is to change variables.
If $p_1(x) $ is positive, the change of variables, 
\begin{equation} \label{chofvar}
w = C_1 \int_0^z \frac{1}{p_1(s)} \ ds, \quad C_1 = 2\pi / \int_0^{2\pi} \frac{1}{p_1(s)} \ ds .
\end{equation}
transforms the periodic operator $p_1(z) D$ to $C_1d/dw$ and carries $[0,2\pi ]$ to $[0,2\pi ]$.
With some limits, this idea has a complex variables extension \cite[p. 50]{GM}.

\begin{prop} \label{confmap}
Suppose $p_1(z)$ is analytic and $2\pi $ periodic in $\strip $.  If the real part of $1/p_1(z)$ is positive on ${\strip }  $,
then $w(z)$ given in \eqref{chofvar} is injective on $\strip $, and the range 
includes a strip $\strip _r = \{ -r < \Im (w) < r \} $.
\end{prop}   

\begin{proof}
For distinct points $z_1, z_2$ in $\strip  $, let $s(t) = z_1 + t(z_2 - z_1)$ for $0 \le t \le 1$ be the straight line path
from $z_1$ to $z_2$.  Then
\[w(z_2) - w(z_1) = C_1 \int_{z_1}^{z_2} \frac{1}{p_1(s)} \ ds 
= C_1(z_2-z_1) \int_0^1 \frac{1}{p_1(s(t))} \ dt . \]
Since the integral has positive real part, $w(z_2) \not= w(z_1)$.

The analytic function $w(z)$ is a $2\pi $ periodic open mapping whose range includes $[0,2\pi ]$,
so by compactness the range includes a strip $\strip _r$.
\end{proof}

In the first order case a simple computation shows that the eigenfunctions of regular operators $D + p_0(z)$ 
have dense span in $\H2 $.  The conformal mapping idea also applies in the second order case, when
\[\dop _2 = p_2(z)D^2 + p_1(z)D + p_0(z).\]
In that case, if the real part of $1/\sqrt{p_2(z)}$ is positive on ${\strip }  $, an application of \propref{confmap} to the new variable
\[w = \int_0^z \frac{1}{\sqrt{p_2(s)} } \ ds \]
reduces the operator to the form $\dop _2 = D^2 + p_1(w)D + p_0(w)$. If 
\[\beta = \exp(- \int_0^w p_1(s)/2 \ ds )\]
is $2\pi $ periodic, the similarity transformation $\beta ^{-1}\dop _2\beta $ 
reduces the operator to Liouville normal form $\dop _2 = D^2 + p_0$
with the same eigenvalues.  
Although the methods are less elementary, perturbation arguments \cite{Carlson79} taking advantage of the 
fact that $D^2$ is self adjoint on $\H2 $ will establish completeness of the eigenfunctions in $\H2 $
for regular operators in Liouville normal form.  This result will also be subsumed by the approach below.

Examples with eigenfunctions which fail to have dense span in $\H2 $
can be found among more general operators $\dop _1$.   To make the computations as transparent as possible
the zeroth order term is dropped, leaving $L_1 = p(z)D$.
The eigenvalue equation
\begin{equation} \label{Hom1}
p(z)Dy = \lambda y ,
\end{equation}
has solutions 
\begin{equation}  \label{efunk1}
y(z,\lambda ) = C \exp (\lambda \int_0^z \frac{1}{p(s)} \ ds ) .
\end{equation}
Assume that $\int_0^{2\pi} 1/p(s) \ ds = 2\pi $.  

The nontrivial solutions of \eqref{efunk1} are $2\pi $ periodic when 
$n$ is an integer and $\lambda $ has one of the values
\begin{equation} \label{evals1}
\lambda _n = in 
\end{equation}
Letting
\[Y(z,\lambda ) = \exp (-\lambda \int_0^z \frac{1}{p(s)} \ ds )\] 
and solving the nonhomogeneous equation $p(z) dy/dz - \lambda y = f(z)$ yields the resolvent formula
\begin{equation} \label{resform1}
R(\lambda ) f = y(z,\lambda ) = Y^{-1}(z,\lambda)[C(\lambda ) + \int_0^z Y(s,\lambda) \frac{f(s)}{p(s)} \ ds ],
\end{equation}
with 
\begin{equation} \label{resconst}
C(\lambda ) = [e^{-2\pi \lambda } - 1]^{-1} \int_0^{2\pi} Y(s,\lambda) \frac{f(s)}{p(s)} \ ds ,
\end{equation}
as long as $\lambda \notin \{ in \}$, which comprises the spectrum of $L_1$.

To account for the $2\pi $ periodicity of functions in $\H2$, consider a fundamental domain
\[\strip _0 = \{ z = x+i \tau  \ | \ 0 \le x < 2\pi, -T < \tau < T \} .\]   
\eqref{efunk1} shows that eigenfunctions will not separate points of $\strip _0$ if the function
\[w(z) = \int_0^z 1/p(s) \ ds \]
is not one-to-one in $\strip _0$.  To construct examples, begin with 
a nonconstant, entire and $2\pi $ periodic function $q(z)$, with $q(x) > 0$ for $x \in \real$.  Take
\[p(z) = C_1\exp (q(z)), \quad C_1 = \frac{1}{2\pi }\int_0^{2\pi } \exp (-q(x)) \ dx ,\]
so that $p(z)$ is $2\pi $ periodic, $w(z+2\pi ) = w(z) + 2\pi $, and $p(z)$ and $w(z)$ are both entire with essential singularities at $\infty $.
The following extension of the Casorati-Weierstrass Theorem is helpful.

\begin{prop} \label{CW}
Suppose $w(z)$ is an entire function with an essential singularity at $\infty $.  For a dense set of points $w_0 \in \complex $, the set
$w_0^{-1} = \{ z \in \complex \ | \ w(z) = w_0 \} $ is infinite. 
\end{prop}

\begin{proof}
Let $U_m = \{ |z| > m \}$ for $m = 1,2,3,\dots $ be neighborhoods of infinity.
Since $w(z)$ is an open mapping,  the images $w:U_m \to \complex $ are both open and dense by the Casorati-Weierstrass Theorem. 
The Baire Category Theorem tells us that $V = \bigcap _m \{ w(U_m) \} $ is dense in $\complex$.  
For any $w_0 \in V$ there is a sequence of distinct points $z_j$ with $w(z_j) = w_0$.
\end{proof}

\begin{prop}
Assume that $w(z)$ is constructed as above.  If $T$ is sufficiently large there will be disjoint nonempty open sets $U_1,U_2 \in \strip _0$
with the property that if $z_1 \in U_1$, then there is a $z_2 \in U_2$ such that
\begin{equation} \label{nosep}
\exp ( inw(z_1)) = \exp( inw(z_2)), \quad n = 0,\pm 1,\pm 2, \dots  .
\end{equation}
That is, the eigenfunctions of $p(z)D$ do not separate $z_1$ and $z_2$.
\end{prop}

\begin{proof}
For $T$ sufficiently large, an application of \propref{CW} guarantees the existence of distinct points $\zeta _1,\zeta _2$ with $w(\zeta _1) = w(\zeta _2)$.
Pick disjoint open neighborhoods $V_j$ of $\zeta _j$.
Since $w(z)$ is an open mapping, ${\cal U} = w(V_1) \cap w(V_2)$ is an open neighborhood of $w(\zeta _1)$.
The sets $w^{-1} ({\cal U}) \cap V_j$ are disjoint open neighborhoods of $\zeta _j$, and for each $\xi _1 \in {\cal U} \cap V_1$
there is a $\xi _2 \in {\cal U } \cap V_2$ such that $w(\xi _1) = w(\xi _2)$.

Since $w(z + 2\pi) = w(z) + 2\pi $, it follows that $\zeta _1 \not = \zeta _2 \mod 2\pi $.
The eigenfunctions $\exp (inw(z))$ are $2\pi $ periodic, so $\zeta _1$, $\zeta _2$, and the associated open neighborhoods may be translated by 
integer multiples of $2\pi $ to $\strip _0$, giving the result.
\end{proof}

\begin{prop}
Assume that $w(z)$ is constructed as above.  If $T$ is sufficiently large 
there is an infinite dimensional subspace ${\cal N}$ of functions $g(z) \in \H2 $ with
\[\langle y_n(z) , g(z) \rangle = 0, \quad g \in {\cal N} \]
for all eigenfunctions $y_n(z)$.
\end{prop}
\begin{proof}
Continuing the argument of the previous proposition, since all of the eigenfunctions of $L_1$ satisfy $y_n(z_1) = y_n(z_2)$,
the uniform approximation of some functions in $\H2 $ by 
linear combinations of eigenfunctions will be impossible.  Moreover  the difference of evaluations $f(z_2) - f(z_1)$ at $z_1$ and $z_2$
 is a continuous functional on $\H2$.  
These functionals must be given by inner products with elements of $\H2 $.  Since the evaluation functionals at $z \in \strip $ are independent,
there is an infinite dimensional subspace $V$ of functions $g(z) \in \H2 $ with
\[\langle y_n(z) , g(z) \rangle = 0, \quad g \in V \]
for all eigenfunctions $y_n(z)$ of $L_1$. 
\end{proof}

One might also consider supplementing the eigenfunctions with generalized eigenfunctions, in particular solutions of $(L_1 - \lambda _nI)^2y = 0$.
If $y(z,\lambda ) = \exp(\lambda w(z))$ is a solution of $p(z)Dy = \lambda y$, then
\[(p(z)D -\lambda )\frac{\partial y}{\partial \lambda } = y.\]
The functions $y_1 = \exp(\lambda w(z))$ and $y_2 = w(z)\exp(\lambda w(z))$ are independent solutions of the equation
$(p(z)D - \lambda )^2 y = 0$.  However the identity $w(z+2\pi ) = w(z) + 2\pi $ implies that   
$y_2(z,\lambda )$ is not $2\pi $ periodic, so $L _1$ has no generalized eigenfunctions that are not already eigenfunctions.
 
\section{Semigroups and completeness of eigenfunctions }

Having displayed examples where completeness of eigenfunctions holds, and others where it fails,  
our efforts now focus on positive results for second order regular differential operators $\dop _A$ with $K \times K$ matrix valued coefficients.
\begin{equation} \label{genopdef}
\dop _A = A_2(z) \frac{d^2}{dz^2}  + A_1(z) \frac{d}{dz} + A_0(z)
\end{equation}
To help with the analysis of $\dop _A$, assume that the eigenvalues of $A_2(z)$ omit the ray $(-\infty ,0]$.
After some preliminary results for operators $\dop _A$, 
the main results will be developed for regular operators $\dop $ which are self adjoint and positive on $\oplus _K \L2per$.
The operators $\dop $  will have eigenfunctions which are complete in $\oplus _K\H2 $.

The completeness proof begins with estimates for the growth of solutions to $\dop _A$ eigenvalue equations in $\overline{\strip }$.
A second step notes that positivity on $\oplus _K \L2per$ implies that for sufficiently small $T > 0$, the operator
$\dop $ is the generator of a holomorphic semigroup on $\oplus _K \H2 $.
We then show that for $t > 0$, the semigroup $S(t)$ is in integral operator whose kernel may be obtained by analytic continuation 
of the eigenfunction expansion on $\oplus _K\L2per$.  
The form of this kernel shows that functions $f$ orthogonal to the eigenfunctions in $\H2 $ must satisfy $S(t)f = 0$ for $t > 0$. 
Since the semigroup is strongly continuous at $t=0$, with $\lim_{t \downarrow 0} S(t)f = f$, it follows that $f = 0$. 

\subsection{Eigenfunction growth in $\strip $}

To obtain growth estimates for solutions of eigenvalue equations for $\dop _A$ in the closed strip $\overline{\strip} $,
reduce the equation $\dop _A Y = \lambda Y$ to a first order system 
\begin{equation} \label{fos}
u'(z,\lambda ) = A(z,\lambda )u(z,\lambda ).
\end{equation}
This involves the standard introduction of the column vector  
\[u(z,\lambda ) = \begin{pmatrix} Y(z,\lambda ), \cr Y'(z,\lambda ) \end{pmatrix},\] 
with $2K$ components.  The coefficient matrix is  
\[A(z,\lambda ) = \begin{pmatrix} 
0_K & I_K  \cr
-A_2^{-1}[A_0 - \lambda I] & -A_2^{-1}A_1  
\end{pmatrix} 
= \begin{pmatrix} 0_K & I_K  \cr
\lambda A_2^{-1} & 0_K  \end{pmatrix} 
- \begin{pmatrix} 0_K & 0_K  \cr
A_2^{-1}A_0 & A_2^{-1}A_1  \end{pmatrix} \] 
  
The existence of a square root for $A_2(z)$ will help simplify the analysis.

\begin{lem}
The $K \times K$ matrix valued function $A_2(z)$ has a square root $A^{1/2}(z)$ which is  
is continuous and invertible on $\overline{\strip }$ and analytic in $\strip $.
\end{lem}  

\begin{proof}
The function $A^{1/2}(z)$ can be defined by a Dunford-Taylor integral \cite[p. 44]{Kato}.
Choose numbers $r_1,r_2$ such that $r_1 < | \mu | < r_2$ for any eigenvalue $\mu $ of  $A_2(0)$.
In addition, choose $0 < \theta < \pi $ such that $-\theta < \arg (\mu ) < \theta $.  
Define a contour $\gamma $ which runs counterclockwise around the circle of radius $r_2$ from $r_2e^{-i\theta }$ to $r_2e^{i\theta }$,
continues with fixed argument to $r_1e^{i\theta }$, then clockwise around the circle of radius $r_1$ to $r_1e^{-i\theta }$, and back with fixed argument to $r_2e^{-i\theta }$.  
The function $\zeta ^{1/2}$ is analytic inside and on $\gamma $ .   
Define $A_2(z)^{1/2}$ for $z$ near $0$ by 
\[A_2(z)^{1/2} = \frac{1}{2\pi i} \int_{\gamma } \zeta ^{1/2} (\zeta - A(z))^{-1} \ d \zeta .\]

By adjusting $r_1$, $r_2$ and $\theta $ when necessary, the function $A_2(z)^{1/2}$ can be continued along a path in the simply connected strip $\strip $.
By the monodromy theorem \cite[p. 307-10]{GK} the extension of $A_2(z)^{1/2}$ is independent of the chosen continuation path.
The desired properties of $A_2(z)^{1/2}$ follow from the integral formula.
\end{proof}
  
A similarity transformation will utilize
\[B(z,\lambda ) = \begin{pmatrix} I_K & -I_K  \cr
 \lambda ^{1/2} A_2^{-1/2} & \lambda ^{1/2}A_2^{-1/2}  \end{pmatrix} , \quad
B^{-1}(z,\lambda ) = \frac{1}{2} \begin{pmatrix} I_K & \lambda ^{-1/2}A_2^{1/2}  \cr
-I_K & \lambda ^{-1/2}A_2^{1/2} \end{pmatrix}.  \]
Note that
\[\begin{pmatrix} 0_K & I_K  \cr
\lambda A_2^{-1} & 0_K  \end{pmatrix} B = B\begin{pmatrix} \lambda ^{1/2}A_2^{-1/2} & 0_K \cr 0_K & -  \lambda ^{1/2}A_2^{-1/2}
\end{pmatrix},\]
so there is a block diagonalization.
\[ \begin{pmatrix} 0_K & I_K  \cr
\lambda A_2^{-1} & 0_K  \end{pmatrix}  = B \begin{pmatrix} \lambda ^{1/2}A_2^{-1/2} & 0_K \cr 0_K & -  \lambda ^{1/2}A_2^{-1/2}\end{pmatrix} B^{-1}.\]

Also,
\[B^{-1} \begin{pmatrix} 0_K & 0_K  \cr
A_2^{-1}A_0 & A_2^{-1}A_1  \end{pmatrix} B
= \frac{1}{2} \begin{pmatrix} \lambda ^{-1/2}A_2^{-1/2}A_0 & \lambda ^{-1/2}A_2^{-1/2}A_1 \cr
 \lambda ^{-1/2}A_2^{-1/2}A_0 & \lambda ^{-1/2}A_2^{-1/2}A_1  \end{pmatrix}B \]
\[ = \frac{1}{2}  \begin{pmatrix}  A_2^{-1/2}[\lambda ^{-1/2}A_0 + A_1A_2^{-1/2}] &  A_2^{-1/2}[- \lambda ^{-1/2}A_0 + A_1A_2^{-1/2}]  \cr
A_2^{-1/2}[ \lambda ^{-1/2}A_0 + A_1A_2^{-1/2}] &  A_2^{-1/2}[- \lambda ^{-1/2}A_0 + A_1A_2^{-1/2}] \end{pmatrix} \]

Rewrite \eqref{fos} as an equation for $w = B^{-1}u$, 
\begin{equation} \label{weqn}
w' =  {\cal A}(z,\lambda) w,
\end{equation}
where 
\[{\cal A}  = \lambda ^{1/2} {\cal A}_2 + {\cal A}_1 + \lambda ^{-1/2}{\cal A}_0,\]
the matrices ${\cal A}_j$ being independent of $\lambda $.

To treat a basis of solutions to \eqref{weqn} simultaneously, 
let $W(z,\lambda )$ be the $2K \times 2K$ matrix function satisfying \eqref{weqn} with the identity matrix $I_{2K}$ as initial value $W(0,\lambda )$.  
Integration of \eqref{weqn} leads to the standard integral equation
\begin{equation} \label{fois}
W(z,\lambda ) = W(0,\lambda ) + \int_0^z {\cal A}(s,\lambda )W(s,\lambda ) \ ds .
\end{equation}
The matrix function
\[U(z,\lambda ) = B(z,\lambda )W(z,\lambda )B^{-1}(0,\lambda ) \]
will then be a $K \times K$ matrix solution of $U' = AU$ with $U(0,\lambda ) = I_K$.

Returning to \eqref{fois},
fix $\zeta \in \overline{\strip} $ with $0 \le \Re (\zeta ) \le 2\pi $ and integrate along the path $z(t) = t\zeta $, for $0 \le t \le 1$.  Along this line \eqref{fois} may be expressed as 
\[W(z(t),\lambda ) = W(0,\lambda ) + \int_0^t {\cal A}(z(s),\lambda )W(z(s),\lambda ) \ \zeta \ ds .\]
With the usual Euclidean norm $\| X \| _E$ for $X \in \complex ^{2K}$ and the matrix norm
\[ \| {\cal A} \| = \sup_{\| X \| _E = 1} \| {\cal A}X \| _E\]
for $2K \times 2K$ matrices ${\cal A}$, the integral equation gives
\[\| W(z(t),\lambda ) \| \le \| I_{2K} \| + \int_0^t \| {\cal A}(z(s),\lambda ) \| \| W(z(s),\lambda ) \| |\zeta | \ ds.\]
Then Gronwall's inequality \cite[p. 241]{CoddCar} gives
\begin{equation} \label{gbnd}
\| W(z,\lambda ) \| \le  \exp( \int_0^1 \| {\cal A}(z(s),\lambda ) \| |\zeta | \ ds ) .
\end{equation}

Since $\| {\cal A} \|  \le  |\lambda |^{1/2} \| {\cal A}_2 \| + \| {\cal A}_1 \| + |\lambda |^{-1/2}\| {\cal A}_0 \| $,
there will be constants $C_0$ and $C_1$ such that $\| {\cal A}(z,\lambda ) \| \le C_0 | \lambda |^{1/2} + C_1$. 
Combining  this estimate for $\| {\cal A} \|$ with $U(z,\lambda ) = B(x,\lambda )W(x,\lambda )B^{-1}(0,\lambda )$ leads to with the following lemma.

\begin{lem} \label{growest}
There are constants $C_2,C_3$ such that the solution of the initial value problem
\[U'(z,\lambda ) = A(z,\lambda )U(z,\lambda ), \quad U(0,\lambda ) = I_{2K},\]
satisfies $\| U(z,\lambda ) \| \le C_2 \| \exp( C_3 | \lambda |^{1/2} ) $ uniformly in $\overline{\strip } \cap \{ 0 \le \Re (z) \le 2\pi \} $.  
\end{lem}

The nonhomogeneous equation
\begin{equation} \label{NH}
A_2(z) \frac{d^2}{dz^2}Y  + A_1(z) \frac{d}{dz}Y + (A_0(z) - \lambda I)Y = f
\end{equation}
may also be reduced to a first order system
\begin{equation}
V'(z,\lambda ) = A(z,\lambda )V + F(z), \quad F(z) = \begin{pmatrix} 0_K \cr A_2^{-1}f \end{pmatrix}.
\end{equation}
Using the basis $U(z,\lambda )$ for the homogeneous equation, the variation of parameters method \cite[p. 33]{CoddCar}, \cite[p. 74-75]{CL} gives solutions
\begin{equation}
V(z,\lambda ) = U(z,\lambda )\xi + U(z,\lambda ) \int_0^z U^{-1}(s,\lambda ) F(s) \ ds , \quad V(0,\lambda ) = \xi .
\end{equation}
The initial value $\xi $ giving a $2\pi $ periodic solution is 
\[\xi = [I_{2K} - U(2\pi ,\lambda )]^{-1}  U(2\pi ,\lambda ) \int_0^{2\pi } U^{-1}(s,\lambda ) F(s) \ ds . \]
That is, \eqref{NH} has a unique $2\pi $ periodic solution 
if and only if $1$ is not an eigenvalue of the (Floquet or monodromy) matrix $U(2\pi , \lambda )$,
which thus characterizes the set of eigenvalues for $\dop _A$.
The resolvent set of $\dop _A$ is the complement of the set of eigenvalues. 

To obtain uniform estimates for normalized eigenfunctions, the leading coefficient
of $\dop _A$ is assumed to have positive real part on the real axis.  Since $\dop _A$ is regular,
the coefficients may be expressed in the form \eqref{goodform1}.

\begin{thm} \label{ebnds}
Suppose the leading coefficient of the regular operator 
\begin{equation} \label{goodform1}
\dop = -DP_2(z)D + P_1(z)D + P_0(z)
\end{equation}
satisfies 
\begin{equation} \label{posreal}
\Re (V^*P_2(x)V) \ge C_0 \| V  \|_E^2 , \quad V \in \complex ^K, \quad x \in \real , \quad C_0  > 0.
\end{equation}
Assume too that $\psi $ is a periodic eigenfunction for $\dop $ with eigenvalue $\lambda $, normalized in
$\oplus _K \L2per$ so that $\| \psi \| _L = 1$.

Then there are constants $C_1,C_2$ such that for all 
$z \in \strip $ and all eigenvalues $\lambda $,
\begin{equation} \label{strest}
\| \psi (z) \| _E \le C_1 \exp( C_2 | \lambda |^{1/2} ) .  
\end{equation}

\end{thm}

\begin{proof}

On the real axis $\psi $ satisfies the equation
\begin{equation} \label{eval2}
-DP_2(x)D \psi +  P_1(x) D\psi + P_0 \psi = \lambda \psi .
\end{equation}
Multiply by $\psi ^*$ and integrate by parts, taking advantage of the periodicity of $\psi $,  to get
\[ \frac{1}{2\pi } \int_0^{2\pi } D\psi ^* P_2 D\psi  \ dx = \lambda - \frac{1}{2\pi } \int_0^{2\pi } \psi ^*P_0\psi + \psi ^*P_1D\psi  \ dx .\]
Taking the real part and using \eqref{posreal}, the Cauchy-Schwarz inequality gives 
\[ \| D\psi \|_L^2 \le c_2(|\lambda | + 1) + c_3 \| D\psi \| _L,\]
so, after completing the square, there is a constant $c_4$ such that
\[ \| D\psi \|_L^2 \le c_4(|\lambda | + 1).\]

Since $\| \psi \| _{L} = 1$, there is a point $x_0 \in [0,2\pi ]$ with $\| \psi (x_0) \| ^2_E \le 1$.
For any $x \in [0,2\pi]$, the fundamental theorem of Calculus and the Cauchy-Schwarz inequality give
\[|\psi ^*\psi (x) - \psi ^*\psi (x_0)| = |\int_{x_0}^x (D\psi )^*\psi (s) + \psi ^*(D\psi (s)) \ ds | \le 4\pi \| D\psi  \|_L  \| \psi \|_L ,\]
leading immediately to a pointwise estimate
\[\psi ^*\psi (x) \le c_0(|\lambda |^{1/2} + 1) , \quad 0 \le x \le 2\pi . \]

Similarly, the estimate $| D\psi ^* D\psi (x_1) | _E \le c_2[|\lambda |+1]$ will hold at some point $x_1 \in [0,2\pi ]$.
For any $x \in [0,2\pi]$, 
\[|D\psi ^*D\psi (x) - D\psi ^*D\psi (x_1)| = |\int_{x_1}^x D^2 \psi ^*D\psi (s) + D\psi ^*D^2\psi (s) \ ds |.\]
Using the eigenvalue equation \eqref{eval2} to replace the second derivatives, and invoking the 
Cauchy-Schwarz inequality again, leads to a pointwise estimate
\[D\psi ^*D\psi (x) \le c_1(|\lambda |^{3/2} + 1), \quad 0 \le x \le 2\pi . \]

The vector function 
\[\begin{pmatrix}\psi (x,\lambda ) \cr \psi '(x,\lambda ) \end{pmatrix}\] 
can be expressed as a linear combination of the columns of the matrix function
$U(z,\lambda )$ from \lemref{growest}.  The pointwise estimates for $\| \psi \| _e^2$ and $\| D\psi \| _E^2$
bound the coefficients of the linear combination by $c(|\lambda |^{3/4} + 1)$, so \lemref{growest}
gives  \eqref{strest}.

\end{proof}

In addition to the earlier hypotheses, the operator $\dop $ is now assumed to be self adjoint as an operator on $\oplus _K \L2per$.
The positivity condition \eqref{postrip} then implies that $\dop + \mu I_K$ is positive for $\mu > 0$ and sufficiently large.
The addition of a constant multiple of the identity will not affect eigenfunction completeness, so 
for simplicity, assume that 
\begin{equation} \label{pos2}
\langle \dop f,f \rangle _L \ge \| f \| ^2 _L 
\end{equation}
We want to take advantage of some coarse estimates \cite{Carlson79} for the eigenvalues $\{ \lambda _n, n = 1,2,3, \dots \} $ of a positve selfadjoint 
regular operator $\dop $, listed in increasing order with multiplicity.  

\begin{lem} \label{szcmp}
Suppose $\dop$ is regular, selfadjoint, and satisfies \eqref{pos2} on $\L2per$, 
with leading coefficient satisfying
\begin{equation} \label{posline}
V^*P_2(x)V \ge C_0 \| V  \|_E^2 , \quad V \in \complex ^K, \quad C_0  > 0, \quad 0 \le x \le 2\pi  .
\end{equation}

There are constants $\alpha _1,\alpha _2 > 0$ such that for all $f$ in the domain of $-D^2$,
\begin{equation} \label{sizecmp}
\alpha _1 \| \dop f \| \le \| (-D^2 + I_K)f \| \le \alpha _2 \| \dop f \| .
\end{equation}

For some $\beta _1,\beta _2 > 0$ and all sufficiently large integers $n$  
the eigenvalues $\lambda _n$ of $\dop $ satisfy 
\begin{equation} \label{betaineq}
\beta _1^2 n^2 \le \lambda _n \le \beta _2^2 n^2 .
\end{equation}

\end{lem}
 \begin{proof}
 
Using \eqref{domdef} and \eqref{posline} we see that \eqref{sizecmp} holds.

Considered on $\oplus _K \L2per$, the operator $-D^2$ has eigenvalues $0$, with multiplicity $K$, and $n^2$ with multiplicity $2K$ for
$n = 1,2,3, \dots $.   For any positive integer $n$, the number of eigenvalues of $-D^2$ which are less than or equal to $n^2$ is less than
$2K(n+1)$.  

Suppose that $\alpha _2 \lambda _{2K(n+1)} < n^2$ for some $n$.
Then there is a norm $1$ eigenfunction $\psi $ for $\alpha _2\dop $ with eigenvalue less than $n^2$ which is orthogonal to
all eigenfunctions of $-D^2$ with eigenvalues at most $n^2 $.  This would give
\[ \|\alpha _2\dop \psi \|  < n^2  , \quad \| -D^2 \psi \| \ge n^2 ,\]
contradicting \eqref{sizecmp}.  

Thus $\lambda _{2K(n+1)} \ge n^2/\alpha _2$ for $n = 1,2,3,\dots $.  Every positive integer $m$ satisfies $2K(n+1) \le m < 2K(n+2)$ for some integer $n$, so
\[\lambda _m \ge n^2/\alpha _2 \ge \frac{1}{\alpha _2} (\frac{m}{2K} - 2)^2.\]
giving  $\lambda _m \ge \beta _1m^2$ for $m$ sufficiently large.

The other comparison in \eqref{betaineq} is similar.
 
\end{proof}

\subsection{Semigroup kernels and eigenfunction completeness}

Recall that a strongly continuous semigroup $S(t)$ of bounded operators on a Banach space is holomorphic if there is a $\delta $ with 
$0 < \delta < \pi/2$ such that $S(t)$ (with its semigroup properties) extends to a holomorphic bounded operator valued function $S(\zeta )$ in the sector $\arg (\zeta ) < \delta $.
Details may be found in \cite[p. 489-493]{Kato}, \cite[p. 60-68]{Pazy}, or \cite[p. 248-253]{RS2}. 
If the real part positivity of the leading coefficient $P_2(z)$ extends to $\overline{\strip }$, then
$\dop $ generates a holomorphic semigroup on $\H2$.

\begin{prop}
Suppose the leading coefficient of the regular operator 
\begin{equation} \label{goodform}
\dop = -DP_2(z)D + P_1(z)D + P_0(z)
\end{equation}
satisfies 
\begin{equation} \label{postrip}
\Re (V^*P_2(z)V) \ge C_0 \| V  \|_E^2 , \quad V \in \complex ^K, \quad C_0  > 0, \quad z \in \overline{\strip} .
\end{equation}

Then $\dop $ generates a holomorphic semigroup $S(t)$ in $\oplus _K \H2 $. 
\end{prop}

\begin{proof}
For $f $ in the domain of $\dop $, integrate the leading term by parts to get the associated form 
\[\langle \dop  f,f \rangle _H = \int_{\BS} (Df)^* P_2(z) Df + f^*[P_1Df + P_0f].\]
Use the Cauchy-Schwarz inequality to get the bound
\[|\int_{\BS} f^*P_1Df| \le C \| f \|_{H} \| Df \| _H \le C(\epsilon \| Df \| _H^2 + \frac{1}{\epsilon } \| f \| _H^2), \quad \epsilon > 0 .\]
Combine this estimate with the positivity condition \eqref{postrip} and the continuity of the coefficients on ${\overline \strip}$;
for sufficiently large positive constants $\mu $, the form $\langle (\dop +\mu I_K) f,f \rangle _H $  takes values in a sector 
\[|\Im (\langle (\dop +\mu I_K)f,f \rangle _H)| \le \gamma \Re (\langle (\dop +\mu I_K)f,f \rangle _H)  ,\quad \gamma > 0.\]

Since, as noted earlier, points the spectrum of $\dop $ are the eigenvalues, the resolvent set of $\dop +\mu I$ includes the left half plane.
The operator $\dop $ is thus $m$-sectorial \cite[p. 279-280]{Kato} and so generates a holomorphic semigroup \cite[p. 492]{Kato} on $\oplus _K \H2$.
For sufficiently large constants $\mu > 0$, the semigroup $S(t)$ generated by $\dop + \mu I$ will be a semigroup of contractions 
\end{proof}

Completeness of the eigenfunctions will now follow from an expansion of the $\oplus _K \L2per$ kernel for $S(t) = \exp (-t\dop )$ in eigenfunctions,
followed by an analytic continuation of this kernel to $\strip $.

\begin{thm} \label{mainthm}
Suppose the regular operator $\dop $ of \eqref{goodform} is self adjoint and positive as an operator on $\oplus _K \L2per$
and satisfies \eqref{postrip} on $\strip $.
Then the periodic eigenfunctions of $\dop $ have dense span in $\oplus _K \H2 $.
\end{thm}

\begin{proof}
As an operator on  $\oplus _K \L2per$, $\dop $ is self adjoint and positive, so generates a holomorphic semigroup $S(t)$ of contractions there.
Since $\dop $ has compact resolvent there is an orthonormal basis of eigenfunctions $\psi _n$ with eigenvalues $\lambda _n$.
As a semigroup on $\oplus _K \L2per$, $S(t)$ can  be represented using an eigenfunction expansion,
 
\begin{equation} \label{semix}
S_L(t)f =\exp (-t \dop ) f(x) = \sum_{n} c_n \exp (-t \lambda _n) \psi _n
\end{equation}
where
\[ c_n = \langle f, \psi _n \rangle _L= \frac{1}{2\pi } \int_0^{2\pi} \psi _n(s)^* f(s) \ ds .\]

\thmref{ebnds} implies that $|\psi _n(z)| \le C_1\exp(C_2\lambda _n^{1/2})$ for $z \in \strip $, 
and so there is a pointwise estimate
\[\| e^{-t\lambda _n} \psi _n(z) \| _E \le C_1\exp(C_2 \lambda _n^{1/2} - t \lambda _n).\] 
\lemref{szcmp} implies that $\lambda _n^{1/2} \ge \alpha n$ for some $\alpha > 0$.
Thus for $t > 0$ the series in \eqref{semix} converges uniformly to a function analytic in $\strip $ and
continuous on $\overline{\strip }$. 

Now suppose $f \in \oplus _K \H2 $, and $S(t)$ is the $\oplus _K \H2 $ semigroup generated by $\dop $.
Let $g(t) = S(t)f$.  For $t > 0$ the function $g(t)$ satisfies the equation
\begin{equation} \label{Cauchy}
\frac{dg}{dt} = \dop g
\end{equation}
in $\oplus \H2 $, and in particular in $\oplus _K \L2per$. 
Since the solutions of \eqref{Cauchy} are uniquely \cite[p. 104]{Pazy} given by $S_L(t)f$,
$g(t) = S_L(t)f = S(t)f$ in $\oplus _K \L2per$.
For $t > 0$ both $g(t) = S(t) f$ and $S_L(t)f$ have analytic continuations to $\strip $,
so $g(t)$ is given by the series representation \eqref{semix}.

Finally, suppose $h \in \oplus _K \H2 $ is othogonal to all eigenfunctions of $\dop $.
Then
\[\langle S(t)f,h \rangle _H =  \langle \sum_{n} c_n \exp (-t \lambda _n) \psi _n, h \rangle = 0, \quad t > 0 .\]
But $S(t)f$ converges in $\oplus _K \H2 $ to $f$ for all $f$ in $\oplus _K \H2 $, implying that $h = 0$.
Consequently, the periodic eigenfunctions of $\dop $ have dense span in $\oplus _K \H2 $.

\end{proof}

\bibliographystyle{amsalpha}

\begin{thebibliography}{10}

\bibitem{BK}
A.~Bakan and S.~Kaijser.
\textit{Hardy spaces for the strip.}
J. Math Anal Appl 333 (2007), pp. 347-364. 

\bibitem{Birkhoff}
G.D.~Birkhoff.
\textit{Boundary value and expansion problems of ordinary linear differential equations.}
Transactions of the AMS, vol. 9, (1908), pp. 373--395. 

\bibitem{Carlson79}
R.~Carlson.
\textit{Expansions associated with non-self-adjoint boundary-value problems.}
Proceediings of the AMS, vol. 73, no.2 (1979), pp. 173--179. 

\bibitem{Carlson80}
R.~Carlson.
\textit{Local completeness for eigenfunctions of regular maximal ordinary differential operators.}
Proceediings of the AMS, vol. 79, no.3 (1980), pp. 400--404. 

\bibitem{CoddCar}
E.~Coddington and R.~Carlson.
\textit{Linear Ordinary Differential Equations}.
SIAM, 1997.

\bibitem{CL}
E.~Coddington and N.~Levinson.
\textit{Theory of Ordinary Differential Equations}.
McGraw-Hill, 1955.

\bibitem{Davies01}
E.B..~Davies.
\textit{Eigenvalues of an elliptic system.}
Math. Z., vol. 243, no.4 (2001), pp. 719--743. 

\bibitem{Davies07}
E.B..~Davies.
\textit{Linear Operators and their Spectra.}
Cambridge University Press (2007)

\bibitem{DS2}
N.~Dunford and J.~Schwartz.
\textit{Linear Operators part II Spectral Theory.}
Wiley-Interscience (1988).

\bibitem{DS3}
N.~Dunford and J.~Schwartz.
\textit{Linear Operators part III Spectral Operators.}
Wiley-Interscience (1988).

\bibitem{Duren}
P.~Duren.
\textit{Theory of $H^p$ Spaces}.
Dover, 2000.

\bibitem{G3}
C.~Gal, S. Gal, and J. Goldstein.
\newblock {\em Evolution Equations with a Complex Spatial Variable}.
\newblock World Scientific, 2014.

\bibitem{Garnett}
J.~Garnett.
\newblock {\em Bounded Analytic Functions}.
\newblock Springer, New York, 2007.

\bibitem{GM}
J.~Garnett and D.~Marshall.
\newblock {\em Harmonic Measure}.
\newblock Cambridge University Press, New York, 2005.

\bibitem{GK}
R.~Greene and S.~Krantz.
\newblock {\em Function Theory of One Complex Variable}.
\newblock American Mathematics Society, 2006.


\bibitem{Hoffman}
K.~Hoffman.
\textit{Banach Spaces of Analytic Functions}.
Prentice-Hall, 1962.

\bibitem{Kato}
T.~Kato.
\textit{Perturbation Theory for Linear Operators}.
Springer-Verlag, New York, 1995.

\bibitem{Pazy}
A.~Pazy.
\newblock {\em Semigroups of Linear Operators and Applications to Partial Differential Equations}.
\newblock Springer, New York, 1983.

\bibitem{RS2}
M.~Reed and B.~Simon.
\textit{Methods of Modern Mathematical Physics 2: Fourier Analysis, Self Adjointness}
Academic Press (1975).

\bibitem{Sarason}
D.~Sarason.
\textit{The $H^p$ spaces of an annulus.}
Mem. Amer. Math. Soc. 56 (1965).

\bibitem{Seeley}
R.~Seeley.
\textit{A  simple example of spectral pathology for differential operators.}
Comm. in Partial Differential Equations, 11 (1986), pp. 595--598. 


\bibitem{Tadmor}
E.~Tadmor.
\textit{The exponential accuracy of Fourier and Chebyshev differencing methods.}
SIAM Journal on Numerical Analysis, 23 (1986), pp. 1--10. 

\bibitem{Tref}
L. ~Trefethen and M.~Embree.
\textit{Spectra and Pseudospectra: The Behavior of Nonnormal Matrices and Operators}.
Princeton University Press (2005)

\bibitem{Villone}
A.~Villone.
\textit{Self-adjoint differential operators.}
Pacific Journal of Mathematics 35 (1970), pp. 517--531. 


\end{thebibliography}

\end{document}